\documentclass[11pt,a4paper]{article}

\usepackage[T1]{fontenc}
\usepackage[utf8]{inputenc}
\usepackage{amsmath,amssymb,amsthm, latexsym,color,epsfig,enumerate,a4, graphicx}
\usepackage{xspace}
\usepackage{bbding}
\usepackage{tikz}
\usetikzlibrary{patterns}
\usepackage{comment}
\usepackage{enumerate}
\usepackage{lmodern}
\usepackage{microtype}
\usepackage{a4wide}

\usepackage{appendix}

\theoremstyle{plain}
\newtheorem{theorem}{Theorem}[section]
\newtheorem{lemma}[theorem]{Lemma}

\newtheorem{observation}[theorem]{Observation}

\newtheorem{remark}[theorem]{Remark}

\theoremstyle{definition}

\newcommand{\eps}{\ensuremath{\varepsilon}}

\date{}
\title{Triangle-factors in pseudorandom graphs}
\author{	
	Rajko Nenadov
	\thanks{
		Department of Computer Science, ETH Zurich, Switzerland. Email: {\tt rnenadov@inf.ethz.ch.}
	}
}

\begin{document}
\maketitle

\begin{abstract}
	We show that if the second eigenvalue $\lambda$ of a $d$-regular graph $G$ on $n \in 3 \mathbb{Z}$ vertices is at most $\eps d^2/(n \log n)$, for a small constant $\eps > 0$, then $G$ contains a triangle-factor. The bound on $\lambda$ is at most an $O(\log n)$ factor away from the best possible one: Krivelevich, Sudakov and Szab\'o, extending a construction of Alon, showed that for every function $d = d(n)$ such that $\Omega(n^{2/3}) \le d \le n$ and infinitely many $n \in \mathbb{N}$ there exists a $d$-regular triangle-free graph $G$ with $\Theta(n)$ vertices and $\lambda = \Omega(d^2 / n)$. 
\end{abstract}

\section{Introduction}

Let $H$ be a graph on $h$ vertices. We say that a graph $G$ on $n \in h \mathbb{Z}$ vertices has an \emph{$H$-factor} if it contains a family of $n/h$  vertex disjoint copies of $H$. For example, in the case where $H = K_2$ is just an edge, a graph has an $H$-factor if and only if it contains a perfect matching. Thus $H$-factors are natural generalisations of perfect matchings from edges to arbitrary graphs. Usually we think of $H$ as being a small (fixed) graph while the number of vertices of $G$, denoted by $n$ throughout the paper, grows.

Determining sufficient conditions for the existence of $H$-factors is a fundamental line of research in Extremal Graph Theory. Textbook theorems of Hall and Tutte give sufficient conditions for the existence of a perfect matching. The first result which treats a more general case is by Corr\'adi and Hajnal \cite{corradi1963maximal} from the 1960s. It shows that any graph with $n \in 3 \mathbb{Z}$ vertices and minimum degree of at least $2n/3$ contains a triangle-factor. This was extended to arbitrary complete graphs by Hajnal and Szemer\'edi \cite{hajnal70}: $\delta(G) \ge (r-1)n/r$ suffices to guarantee a $K_r$-factor for any $r \ge 3$. This bound is easily seen to be the best possible. The question of characterising the smallest possible  minimum degree for an arbitrary graph $H$ attracted significant attention \cite{alon1996h,komlos2000tiling,komlos2001proof} until it was finally resolved by K\"uhn and Osthus \cite{kuhn2009minimum}. Various other versions of the problem that put further restrictions on a graph have been studied, with the common goal to reduce sufficient minimum  degree and therefore the minimum number of edges of graphs that satisfy it. These include multipartite \cite{fischer1999variants,keevash15,keevash2015multipartite,lo2013multipartite,martin2017asymptotic2} and Ramsey-T\'uran versions \cite{balogh2018triangle2,balogh2016triangle1}. However, all these results require $\delta(G) \ge f(H) n$ for some function $f$ depending on particular restrictions we put on a graph $G$. In particular, this does not tell us anything about graphs with $o(n^2)$ edges. What conditions for the existence of an $H$-factor could also be satisfied by sparse graphs?

This question was first addressed by Krivelevich, Sudakov and Szab\'o \cite{krivelevich2004triangle}. They proposed that, in the case of regular graphs, the role of the minimum degree condition in the Corr\'adi-Hajnal theorem can be replaced by a bound on the \emph{second eigenvalue}. Given a $d$-regular graph $G$, let $A = A(G)$ be its adjacency matrix and $\lambda_1 \ge \lambda_2 \ge \ldots \ge \lambda_n$ be its eigenvalues. Then $\lambda(G)$, the so-called second eigenvalue, is defined as $\lambda(G) = \max\{|\lambda_2|, |\lambda_n|\}$. The \emph{expander mixing lemma} (e.g. see \cite[Corollary 2.5]{alon2004probabilistic}) shows that $\lambda(G)$ governs the edge distribution of $G$. The smaller the $\lambda(G)$ is the more edge distribution of $G$ resembles that of $G(n, d/n)$, the Erd\H{o}s-Renyi random graph with edge probability $p = d/n$. It is convenient to quantify this in terms of \emph{$(\beta, p)$-bijumbledness}: a graph $G$ with $n$ vertices is $(\beta, p)$-bijumbled for some $p \in [0,1]$ and $\beta > 0$ if for every subsets $X, Y \subseteq V(G)$ we have
$$
	\left| e(X, Y) - |X||Y|p \right| \le \beta \sqrt{|X||Y|},
$$
where $e(X, Y)$ denotes the number of edges in $G$ with one endpoint in $X$ and the other in $Y$. Where $X$ and $Y$ are not disjoint, every edge which lies in their intersection is counted twice. 
The expander mixing lemma states that if $G$ is a $d$-regular graph with $n$ vertices then it is $(d/n, \lambda(G))$-bijumbled. A result of Erd\H{o}s and Spencer \cite{erdos1971imbalances} shows that if a graph $G$ is $(p, \beta)$-bijumbled then $\beta = \Omega(\sqrt{np})$,  thus we always have $\lambda(G) = \Omega(\sqrt{d})$. As one would expect, a random graph $G(n,p)$ is typically $(p, \Theta(\sqrt{np}))$-bijumbled. The question now becomes how close a graph needs to be to a random graph, or how close its edge distribution has to be to an  optimal one, in order to have a triangle-factor? Of course, we need that $d$ is not too small as otherwise even an optimal edge distribution would not be sufficient to guarantee a triangle-factor. We will quantify this statement shortly.

Krivelevich, Sudakov and Szab\'o \cite{krivelevich2004triangle} showed that if a $d$-regular graph $G$ with $n \in 3 \mathbb{Z}$ vertices  satisfies $\lambda(G) = o(d^3 / (n^2 \log n))$ then it contains a triangle-factor. As $\lambda(G) = \Omega(\sqrt{d})$, this only applies to $d$-regural graphs for $d = \omega(n^{4/5} \log^{2/5} n)$. An upper bound on $\lambda(G)$ was further relaxed to $\lambda(G) = O(d^{5/2}/n^{3/2})$ by Allen, B\"ottcher, H\`an, Kohayakawa and Person \cite{allen2017powers}. On the other hand, an ingenious construction by Alon \cite{alon1994explicit} (see also \cite{conlon2017sequence}) shows that there are infinitely many $n \in \mathbb{N}$ for which there exists a triangle-free $d$-regular graph $G$ with $n$ vertices, where $d = \Theta(n^{2/3})$ and $\lambda(G) = \Theta(\sqrt{d})$. Such graphs are close to optimal in terms of edge distribution, yet have surprisingly high density for triangle-free graphs. In comparison, any construction of such graphs based on random graphs achieves a density of at most $O(n^{-1/2})$. Krivelevich, Sudakov and Szab\'o \cite{krivelevich2004triangle} used the construction of Alon to show that for any function $d = d(n)$ such that $\Omega(n^{2/3}) \le d \le n$ and infinitely many $n \in \mathbb{N}$ there exists a triangle-free $d$-regular graph $G$ with $\Theta(n)$ vertices and $\lambda(G) = \Omega(d^2 / n)$. It is straightforward to show that if $\lambda(G) \le \eps d^2 / n$ then the graph $G$ contains a triangle (see Section \ref{sec:prelim}). Significantly improving previous bounds on $\lambda$, our main result shows that this is almost sufficient for the existence of a triangle-factor.

\begin{theorem} \label{thm:main}
	There exists $\eps > 0$ such that if $G$ is a $d$-regular graph on $n \in 3 \mathbb{N}$ vertices with the second eigenvalue $\lambda < \eps d^2 / (n \log n)$ then $G$ contains a triangle-factor.
\end{theorem}


On the one hand, Theorem \ref{thm:main} states that if $d = \Theta((n \log n)^{2/3})$ then every optimal (with respect to edge distribution) $d$-regular graph with $n$ vertices contains a triangle-factor. On the other hand, for $d = O(n^{2/3})$ there exists at least one optimal $d$-regular graph with $n$ vertices which does not contain even a single triangle. There is an analogy between this result and, say, a similar result for random graphs: for $p = cn^{-1/2}$ we have that $G(n,p)$ is `almost' triangle-free (that is, all triangles can be destroyed by removing a negligible number of edges), while for $p = C n^{-1/2}$ not only can one not make it triangle-free, but even after removal of up to a third of the edges touching each vertex it contains an `almost' triangle-factor (see \cite{balogh2012corradi}). 

As we only rely on the edge distribution of a graph $G$, instead of proving Theorem \ref{thm:main} directly, we prove the following theorem for a broader class of bijumbled graphs. By the expander mixing lemma this implies Theorem \ref{thm:main}.

\begin{theorem} \label{thm:bijumbled}
	There exists $\eps > 0$ such that if $G$ is a $(p, \eps np^2/\log n)$-bijumbled graph on $n \in 3 \mathbb{N}$ vertices and minimum degree at least $np/2$, then $G$ contains a triangle-factor.
\end{theorem}

The definition of bijumbled graphs does not put any restriction on the minimum degree and, indeed, it can contain isolated vertices. The minimum degree $np/2$ in Theorem \ref{thm:bijumbled} can be replaced by $\alpha np$ for any constant $\alpha > 0$ (influencing $\eps$). We could also prove Theorem \ref{thm:bijumbled} for jumbled graphs, however as the notion of bijumbledness is somewhat easier to deal with, we have decided not to overwhelm the reader with straightforward but technical details. 

We emphasise again that $\beta = \Omega(\sqrt{np})$ has the following two implications for Theorem \ref{thm:bijumbled}: (i) $p = \Omega(n^{-1/3})$ and (ii) by choosing $\eps > 0$ to be sufficiently small we can assume that $n$ is as large as we want. If either of these conditions is not met then there are no $(p, \eps np^2 / \log n)$-bijumbled graphs  and the theorem vacuously holds. Therefore for the rest of the paper we assume that $n$ is sufficiently large.

\section{Preliminaries} \label{sec:prelim}

Throughout the paper we use the term \emph{disjoint} as a shorthand for \emph{vertex-disjoint}.

Given a graph $G$, we denote with $\tau(G)$ the size of a smallest vertex cover of $G$. Note that if $G$ and $G'$ are graphs on the same vertex set, then $\tau(G \cup G') \le \tau(G) + \tau(G')$. The following generalisation of Hall's theorem by Haxell \cite{haxell1995condition} has turned out to be invaluable tool for embedding spanning graphs in (pseudo-)random graphs. We  repeatedly apply it in the proof of Theorem \ref{thm:bijumbled}.

\begin{theorem} \label{thm:hax}
	Let $G_1, \ldots, G_\ell$ be a family of graphs on the same vertex set. If for every $I \subseteq [\ell]$ we have $\tau(\bigcup_{i \in I} G_i) \ge 3 |I|$ then we can choose an edge $e_i$ from each graph $G_i$ such that all these edges are pairwise disjoint.
\end{theorem}

In the rest of the section we collect some properties of $(p, \beta)$-bijumbled graphs.

\begin{lemma} \label{lemma:edge}
	Let $G$ be a $(p, \beta)$-bijumbled graph. Then for any two  subsets $X, Y \subseteq V(G)$ such that $|X|, |Y| > \beta / p$ there exists an edge with one endpoint in $X$ and the other in $Y$.
\end{lemma}
\begin{proof}
	The claim follows directly from the definition of $(p, \beta)$-bijumbledness and lower bounds on $|X|$ and $|Y|$:
	$$
		e(X, Y) \ge |X||Y|p - \beta \sqrt{|X||Y|} = \sqrt{|X||Y|} (\sqrt{|X||Y|} p - \beta) > 0. 
	$$	
\end{proof}

\begin{lemma} \label{lemma:small_degree}
	Let $G$ be a $(p, \beta)$-bijumbled graph. Then for any $\gamma > 0$ and any subset $W \subseteq V(G)$ there are at most 
	$$
		\frac{\beta^2}{\gamma^2 p^2 |W|}
	$$
	vertices with less than $(1 - \gamma)|W|p$ neighbours in $W$.
\end{lemma}
\begin{proof}
	Consider a subset $W \subseteq V(G)$ and let $X \subseteq V(G)$ be the subset consisting of all vertices with less than $(1 - \gamma)|W|p$ neighbours in $W$. Then
	$$
		(1 - \gamma)|X||W|p > e(X, W) \ge |X||W|p - \beta \sqrt{|X||W|},
	$$
	implying the desired bound on $X$.	
\end{proof}

\begin{lemma} \label{lemma:triangle}
	Let $G$ be a $(p, \beta)$-bijumbled graph. Then for every three  subsets $X, Y, Z \subseteq V(G)$ (not necessarily disjoint) of size $|X| \ge 1 + 4\beta$ and $|Y|, |Z| \ge 2\beta / p^2$ there exists a triangle in $G$ with one vertex in each $X$, $Y$, and $Z$.
\end{lemma}
\begin{proof}
	By Lemma \ref{lemma:small_degree} there are at most $\frac{\beta^2}{(1/2)^2 p^2 |Y|} \le 2\beta$ vertices with less than $|Y|p/2$ neighbour in $Y$, and similarly at most $2 \beta$ vertices with less than $|Z|p/2$ neighbours in $Z$. Therefore there exists a vertex in $X$ with at least $|Y|p/2 \ge \beta / p$ and $|Z|p/2 \ge \beta / p$ neighbours in $Y$ and $Z$, respectively. By Lemma \ref{lemma:edge} there exists an edge between these sets, thus giving a desired triangle.
\end{proof}

\section{The main building block: $K_4^-$ chains}

Let $K_4^-$ be a graph obtained from the complete graph on $4$ vertices by removing a single edge. We define an $\ell$-chain as a graph obtained by sequentially `gluing' $\ell$ copies of $K_4^-$ on vertices of degree $2$ (see Figure \ref{fig:chain}). If the length is not important we simply refer to it as a chain. Note that an $\ell$-chain contains exactly $\ell + 1$ vertices such that removal either of them (but exactly one!) results in a graph that has a triangle-factor (square vertices in Figure \ref{fig:chain}). We call these vertices \emph{removable}. For simplicity we define a $0$-chain to be a single vertex. If a graph $H$ is isomorphic to an $\ell$-chain then we denote with $R(H)$ the set of removable vertices in $H$. 

\begin{figure}[h!] \label{fig:chain}		
	\centering
	\begin{tikzpicture}[scale = 0.6]
		\tikzstyle{blob} = [fill=black,circle,inner sep=1.7pt,minimum size=0.5pt]
		\tikzstyle{sq} = [fill=black,rectangle,inner sep=3.5pt,minimum size=2.5pt]

		\node[sq] (f1) at (0, 0) {};
		\node[blob] (v1) at (1, -1) {}; \node[blob] (v2) at (1, 1) {}; 
		\node[sq] (f2) at (2, 0) {};
		\node[blob] (v3) at (3, -1) {}; \node[blob] (v4) at (3, 1) {}; 
		\node[sq] (f3) at (4, 0) {};
		\node[blob] (v5) at (5, -1) {}; \node[blob] (v6) at (5, 1) {}; 
		\node[sq] (f4) at (6, 0) {};
		\node[blob] (v7) at (7, -1) {}; \node[blob] (v8) at (7, 1) {}; 
		\node[sq] (f5) at (8, 0) {};

		\draw (f1) -- (v1) -- (f2) -- (v3) -- (f3) -- (v5) -- (f4) -- (v7) -- (f5) -- (v8) -- (f4) -- (v6) -- (f3) -- (v4) -- (f2) -- (v2) -- (f1);
		\draw (v1) -- (v2); \draw (v3) -- (v4); \draw (v5) -- (v6); 
		\draw (v7) -- (v8);

		\node at (6, -0.5) {$v$};

		\draw[pattern=north east lines] (f1.center) -- (v1.center) -- (v2.center) -- (f1.center) -- cycle;	

		\draw[pattern=north east lines] (f2.center) -- (v3.center) -- (v4.center) -- (f2.center) -- cycle;

		\draw[pattern=north east lines] (f3.center) -- (v5.center) -- (v6.center) -- (f3.center) -- cycle;

		\draw[pattern=north east lines] (f5.center) -- (v7.center) -- (v8.center) -- (f5.center) -- cycle;
	\end{tikzpicture}
	\caption{A $4$-chain with a triangle-factor after removing vertex $v$.}	
\end{figure}
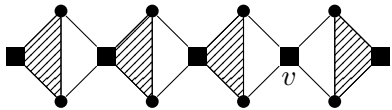

We say that a triangle in $G$ \emph{traverses} some sets $A, B, C \subseteq V(G)$ if it intersects all of them. We repeatedly use the following important observation.

\begin{observation} \label{obs:chain}
	Let $C_1, C_2$ and $C_3$ be disjoint chains in a graph $G$. If there exists a triangle in $G$ which traverses $R(C_1), R(C_2)$ and $R(C_3)$ then the subgraph of $G$ induced by $V(C_1) \cup V(C_2) \cup V(C_3)$ contains a triangle factor.
\end{observation}

The following lemma shows that if copies of $K_4^-$ are scattered throughout the graph then we can find a large $\ell$-chain.

\begin{lemma} \label{lemma:chains}
	Let $G$ be a graph with $n$ vertices such that for every two disjont subsets $X, Y \subseteq V(G)$ of size $|X|, |Y| \ge \alpha n$ there exists a copy of $K_4^-$ in $G$ with one vertex of degree $2$ in $X$ and the other vertices in $Y$. Then $G$ contains an $\ell$-chain for every $1 \le \ell \le (1 - 5 \alpha) n/3$.
\end{lemma}
\begin{remark}
	Graphs similar to $\ell$-chains have been used by Krivelevich \cite{krivelevich1997triangle} in what is probably the first application of the absorbing method in random graphs. Rather than gluing $K_4^-$'s such that they form a chain, Krivelevich uses them to form a tree of an unspecified shape with the property that at least a third of the vertices are removable. We could as well use such graphs instead of chains and some of the lemmas from \cite{krivelevich1997triangle} instead of Lemma \ref{lemma:chains}. However we have decided to spell out the proof and stick to $\ell$-chains for the sake of completeness and brevity.
\end{remark}
\begin{proof}[Proof of Lemma \ref{lemma:chains}]
	It suffices to show that we can find an $\ell$-chain for $\ell = \lfloor (1 - 5 \alpha) n/3 \rfloor$. Consider an exploration of the graph $G$ by building a chain $C$ using the following depth-first search procedure. Initially set $C = \emptyset$, $D, D' = \emptyset$ and $U = V(G)$ and as long as $U \neq \emptyset$ do:
	\begin{itemize}
		\item if $C = \emptyset$ then choose an arbitrary vertex $v \in U$, remove it from $U$ and set $C := \{v\}$;

		\item if $C \neq \emptyset$ and there exists a copy of $K_4^-$ which contains the last added removable vertex in $C$ as a degree-$2$ vertex and has all other vertices in $U$, glue such a copy onto $C$ and remove the vertices from $U$;

		\item otherwise remove the last added removable vertex from $C$ and move it to $D$; if $C$ was not a single vertex then also remove from $C$ the two last added non-removable vertices and add them to $D'$ (see Figure \ref{fig:explore}).
	\end{itemize}

	\begin{figure}[h!] \label{fig:explore}		
	\centering
	\begin{tikzpicture}[scale = 0.4]
		\tikzstyle{blob} = [fill=black,circle,inner sep=1.7pt,minimum size=0.5pt]
		\tikzstyle{dot} = [fill=black,circle,inner sep=0.5pt,minimum size=0.5pt]
		\tikzstyle{sq} = [fill=black,rectangle,inner sep=2.5pt,minimum size=2.5pt]

		\node[sq] (f1) at (0, 0) {};
		\node[blob] (v1) at (1, -1) {}; \node[blob] (v2) at (1, 1) {}; 
		\node[sq] (f2) at (2, 0) {};
		\node[blob] (v3) at (3, -1) {}; \node[blob] (v4) at (3, 1) {}; 
		\node[sq] (f3) at (4, 0) {};

		\node[sq] (f3p) at (7, 0) {};
		\node[blob] (v5) at (8, -1) {}; \node[blob] (v6) at (8, 1) {}; 
		\node[sq] (f4) at (9, 0) {};
		\node[blob] (v7) at (10, -1) {}; 
		\node[blob] (v8) at (10, 1) {}; 

		\node[sq] (f5) at (11, 0) {};

		\node at (5.5, -2) {$C$};

		\draw (f1) -- (v1) -- (f2) -- (v3) -- (f3) -- (v4) -- (f2) -- (v2) -- (f1);
		\draw (f3p) -- (v5) -- (f4) -- (v7) -- (f5) -- (v8) -- (f4) -- (v6) -- (f3p);

		\draw (f3) -- (4.8, 0.8); \draw (f3) -- (4.8, -0.8);
		\draw (f3p) -- (6.2, 0.8); \draw (f3p) -- (6.2, -0.8);

		\node[dot] at (5, 0) {};
		\node[dot] at (5.5, 0) {};
		\node[dot] at (6, 0) {};

		\node[blob] (u1) at (15, -1) {};
		\node[blob] (u2) at (14, -4) {};
		\node[sq] (fu) at (17, -3) {};
		\draw[dashed] (f5) -- (u1) -- (fu) -- (u2) -- (f5);
		\draw[dashed] (u1) -- (u2);
		
		\draw[->] (f5) to [bend left] (9, -5);
		\draw[->] (v7) to [bend left=15] (5, -5);
		\draw[->] (v8) to [bend left=25] (4, -5);

		\draw[rounded corners] (12, 2) -- (18, 2) -- (18, -9) -- (12, -9) -- cycle;
		\node at (12.8, -8.3) {$U$};

		\draw[rounded corners] (7, -4) -- (11, -4) -- (11, -9) -- (7, -9) -- cycle;
		\node at (7.8, -8.3) {$D$};

		\draw[rounded corners] (0, -4) -- (6, -4) -- (6, -9) -- (0, -9) -- cycle;
		\node at (0.8, -8.3) {$D'$};
	\end{tikzpicture}
	\caption{Exploration of a graph $G$.}	
	\end{figure}
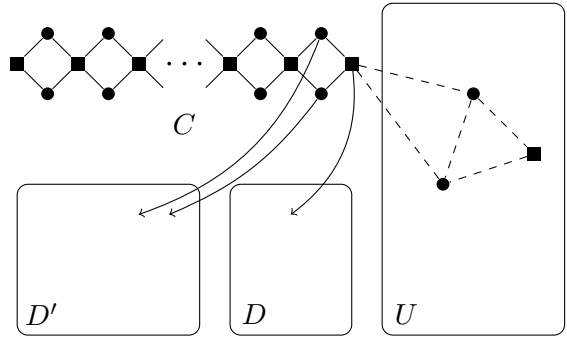

	Note that at every step of the exploration $D'$ is at most twice the size of $D$ and none of the sets changes by more than $3$ vertices. Moreover, by the construction there is no copy of $K_4^{-}$ with a vertex of degree 2 in $D$ and the other three vertices in $U$. If at some point $C$ is an $\ell$-chain then we are done. Otherwise consider the first step after which $|U| \le \alpha n + 3$ and note that then necessarily $|U| \ge \alpha n$. As $|C| < 3\ell + 1 \le (1 - 5 \alpha) n + 1$, we have
	$$
		|D| + |D'| = n - |U| - |C| \ge n - \alpha n - 3 - (1 - 5 \alpha )n - 1 = 4 \alpha n - 4 > 3 \alpha n.
	$$
	From $|D'| \le 2|D|$ we conclude $|D| \ge \alpha n$. However as $|U|, |D| \ge \alpha n$ we obtain a contradiction with the fact that there is no copy of $K_4^-$ with one vertex of degree $2$ in $D$ and the other vertices in $U$. Therefore $G$ contains an $\ell$-chain.
\end{proof}

It is straightforward to show that sufficiently bijumbled graphs satisfy the requirement of Lemma \ref{lemma:chains}, thus we have the following corollary.

\begin{lemma} \label{lemma:chain_jumbl}
	Let $G$ be an $(p, o(np^2))$-bijumbled graph. Then for every $W \subseteq V(G)$ of size $|W| \ge n/4$ the induced subgraph $G[W]$ contains an $\ell$-chain for every $1 \le \ell \le |W|/6$.
\end{lemma}
\begin{proof}
	We only need to check that for every disjoint subsets $X, Y \subset W$ of size $|W|/20$ there exists a copy of $K_4^-$ with one vertex in $X$ and the other three in $Y$. By Lemma \ref{lemma:small_degree} all but $o(n p^2)$ vertices in $G$ have at least  $|Y|p/2$ neighbours in $Y$ and at least $|X|p/2$ neighbours in $X$. Pick one such good vertex $y \in Y$. All we need to do now is find a vertex $y' \in N_G(y) \cap Y$ which has a neighbour in both $N_G(y) \cap X$ and $N_G(y) \cap Y$. Lemma \ref{lemma:small_degree} states that there are at most $o(np)$ vertices in $G$ which do not have such a property, thus as $|Y|p/2 = \Theta(np)$ the set $N_G(y) \cap Y$ contains a desired vertex.
\end{proof}

The following lemma captures the main `absorbing' property of chains. To digest its statement we recommend the reader to see how it is applied in the proof of Theorem \ref{thm:bijumbled}.

\begin{lemma} \label{lemma:cascade}
	Let $G$ be a $(p, \beta)$-bijumbled graph with $n$ vertices for some $\beta = o(np^2)$. Suppose we are given disjoint $\ell$-chains $C'_1, \ldots, C'_t \subseteq G$ for some $t, \ell \in \mathbb{N}$ such that $\ell$ is even, $t \ge 2000$ and
	$400 \beta / p^2 \le t  (\ell+1) \le n/24$. Then for any subset $W \subseteq V(G) \setminus \bigcup_{i \in [t]} V(C_i')$ of size $|W| \ge n/4$ there exist disjoint $(\ell/2)$-chains $C_1, \ldots, C_{2t} \subseteq G[W]$ with the following property: for every $L \subseteq [2t]$ there exists $L' \subseteq [t]$ such that the subgraph of $G$ induced by
	$$
		\bigcup_{i \in L} V(C_{i}) \cup \bigcup_{i \in L'} V(C'_i)
	$$
	contains a triangle-factor.
\end{lemma}
\begin{proof}
	By repeated application of Lemma \ref{lemma:chain_jumbl} we obtain disjoint $(\ell/2)$-chains $C_1, \ldots, C_{3t} \subseteq G[W]$. For each $i \in [3t]$ form a graph $G_i$ on the vertex set $V' = [t]$ by adding an edge $\{j,k\}$ for $1 \le j < k \le t$ iff there exists a triangle which traverses $R(C_i)$, $R(C_j')$ and $R(C_k')$. We first show that there exists a subset $I \subseteq [3t]$ of size $|I| = 2t$ such that for every non-empty $J \subseteq I$ of size $|J| \le t / 12$ we have
	\begin{equation} \label{eq:tau_J}
		\tau(\bigcup_{i \in J} G_i) \ge 3|J|.
	\end{equation}

	Let $q = 0$ and as long as there exists a subset $J \subseteq [3t] \setminus \bigcup_{j = 1}^{q} J_j$ of size $|J| \le t/12$ which violates \eqref{eq:tau_J} set $J_{q+1} := J$ and increase $q$. Let $B = \bigcup_{j = 1}^q J_j$. We claim that $|B| < t/12$. Suppose towards a contradiction that this is not the case and consider the smallest $q' \le q$ such that $B' = \bigcup_{j = 1}^{q'} J_j$ is of size $|B'| \ge t/12$. As $|J_{q'}| \le t/12$ we have have $t/12 \le |B'| \le t/6$. By the choice of $J_1, \ldots, J_{q'}$ we have
	$$
		\tau(\bigcup_{i \in B'} G_i) \le \sum_{z = 1}^{q'} \tau(\bigcup_{i \in J_z} G_i) \le  \sum_{z = 1}^{q'} 3|J_z| = 3|B'| \le t/2.
	$$
	This implies that there exists a subset $I' \subseteq V'$ of size $|I'| \ge t/2$ which is an independent set in every $G_i$ for $i \in B'$ (recall that a complement of a vertex cover is an independent set). Split such $I'$ arbitrarily into two (nearly)-equal subsets $I'_1$ and $I'_2$ and consider sets $Y_1 = \bigcup_{j_1 \in I_1'} R(C'_j)$ and $Y_2 = \bigcup_{j_2 \in I_2'} R(C'_j)$. From assumptions on $t$ and $\ell$ we have 
	$$
		|Y_1|, |Y_2| \ge \lfloor t/4 \rfloor (\ell + 1) > 2 \beta / p^2.
	$$ 
	On the other hand the set $Y = \bigcup_{i \in B'} R(C_i)$ is of size $|Y| \ge \frac{t}{12} (\ell/2 + 1) \ge 5 \beta / p^2 \ge 1 + 4 \beta$. By Lemma \ref{lemma:triangle} there exists a triangle which traverses $Y$, $Y_1$ and $Y_2$, contradicting the assumption that $I'$ is an independent set in $\bigcup_{i \in B_{q'}} G_i$. To conclude, at the end of the procedure we have $|B| \le t/12$  thus we may take $I \subseteq [3t] \setminus B$ to be an arbitrary subset of size $2t$. 

	Relabel $\{C_i\}_{i \in I}$ as $C_1, \ldots, C_{2t}$. It remains to show that these chains have the desired property. Consider a subset $L \subseteq [2t]$. Let $L_1 \subseteq L$ be a maximal subset such that the subgraph induced by 
	$$
		\bigcup_{i \in L_1} V(C_i)
	$$
	contains a triangle-factor and set $L_2 = L \setminus L_1$. Note that $|L_2| \le t/12$: otherwise split $L_2$ into three (nearly-)equal subsets $J_1, J_2$ and $J_3$ and set $Y_k = \bigcup_{i \in J_k} R(C_i)$ for $k \in \{1,2,3\}$. Each $Y_k$ is of size 
	$$
		|Y_k| \ge \lfloor t/36 \rfloor (\ell/2 + 1) \ge 5 \beta / p^2 \ge \max\{1 + 4 \beta, 2 \beta / p^2\},
	$$
	thus again by Lemma \ref{lemma:triangle} there exists a triangle in $G$ which traverses $Y_1, Y_2$ and $Y_3$. Such a triangle then also traverses $R(C_{i_1}), R(C_{i_2})$ and $R(C_{i_3})$ for some  distinct $i_1, i_2, i_3 \in L_2$. Observation \ref{obs:chain} now gives a contradiction with the maximality of $L_1$. 

	From $|L_2| \le t/12$ we have that \eqref{eq:tau_J} holds for every $J \subseteq L_2$. By Theorem \ref{thm:hax} we can then choose an edge $e_i = \{v_i, w_i\} \in G_i$ for each $i \in L_2$ such that these edges are pairwise disjoint. This means that for every $i \in L_2$ there exists a triangle which traverses $R(C_i), R(C_{v_i}')$ and $R(C_{w_i}')$ (recall the construction of $G_i$). Therefore by Observation \ref{obs:chain} we obtain a triangle-factor of the subgraph induced by
	$$
		\bigcup_{i \in L_2} V(C_i) \cup V(C_{v_i}') \cup V(C_{w_i}'),
	$$
	which completes the proof of the lemma.
\end{proof}

\section{Proof of Theorem \ref{thm:bijumbled}}

Our proof strategy is inspired by ideas of Krivelevich \cite{krivelevich1997triangle}. However as we cannot use multiple exposure of random edges, a technique often employed in study of random graphs and heavily used in \cite{krivelevich1997triangle}, new ideas are required. We start by finding $2000$ disjoint $2^q$-chains using Lemma \ref{lemma:chain_jumbl}, where $q = \Theta(\log n)$ is chosen such that $2^q = \Theta(n / \log n)$. Then, iteratively, using Lemma \ref{lemma:cascade} construct additional $q$ levels of chains labelled from $q - 1$ down to $0$,  such that the $i$-th level consists of $2000 \cdot 2^{q - i}$ disjoint $2^i$-chains (recall that $0$-chains are just vertices) and the level $i+1$ has the `absorbing' property for level $i$ (see the property stated in Lemma \ref{lemma:cascade}). Let $U_i$ denote the set of vertices used in the $i$-th level. 

Suppose that there exists a collection of disjoint triangles which cover all the vertices in $V' = V(G) \setminus \bigcup_{i = 0}^q U_i$, some in $U_0$ and none in $U_i$ for $i \ge 1$. This corresponds to finding an almost triangle-factor in $G[V' \cup U_0]$ with a twist -- we need to cover all of $V'$. This turns out to be a straightforward (though non-trivial) task. Having such triangles, let $L_0 \subseteq U_0$ be the subset of vertices which are not used. It remains to find a triangle-factor in the subgraph induced by $L_0 \cup U_1 \cup \ldots \cup U_q$. This is done using a `cascading' effect of the levels: There exists a subset of chains in the $1$-st level which together with $L_0$ contain a triangle-factor. The remaining chains from the $1$-st level can be further covered using some chains from the $2$-nd level, and so on until we cover the remaining chains from the $(q-1)$-th level using the $q$-th level. This leaves us with $3k$ chains in the last level. However as these chains have $\Theta(n / \log n)$ removable vertices, we can simply partition them into groups of three and for each group find a triangle traversing the corresponding removable vertices. By Observation \ref{obs:chain} this finishes a triangle-factor. 

The main novelty compared to the proof by Krivelevich \cite{krivelevich1997triangle} lies in Lemma \ref{lemma:cascade}. Instead of showing that one can choose levels of chains such that each level has an absorbing property with the respect to the one below, as we have done here, Krivelevich uses multiple exposure of edges to absorb the remaining vertices into the current level.

\begin{proof}[Proof of Theorem \ref{thm:bijumbled}]	 
	Consider a random equipartition $V(G) = V_1 \cup V_2$. By Chernoff's inequality and union bound, with positive probability every vertex has at least $np/6$ neighbour in both $V_1$ and $V_2$ (we assumed here that $p \ge n^{-1/3} \gg \log n / n$ and $n$ is sufficiently large). Take one such partition.
	
	Let $q = \lceil \log_2 ( \frac{n}{ 16 \cdot 10^4 \log n} ) \rceil$ and, for convenience, for each $i \ge 0$ set $t_i = 2000 \cdot 2^{q-i}$. By repeatedly applying Lemma \ref{lemma:chain_jumbl} we obtain disjoint $2^q$-chains $C_1^q, \ldots, C_{t_q}^q \subseteq G[V_2]$ and let $U_q = \bigcup_{j = 1}^{t_q} V(C_j^q)$. We now define levels $i = q - 1, \ldots, 0$ iteratively as follows: let $C_1^i, \ldots, C_{t_i}^i \subseteq G[W_i]$ be disjoint $2^i$-chains given by Lemma \ref{lemma:cascade} for $\{C_j^{i+1}\}_{j \in [t_{i+1}]}$ and $W_i = V_2 \setminus \bigcup_{i' > i} U_{i'}$, and set $U_i = \bigcup_{j = 1}^{t_i} V(C_j^i)$. This is indeed possible: simple calculations show that each $U_i$ is smaller than $n/(10 \log n)$ thus $|W_i| \ge n / 4$, with room to spare, and $t_i 2^i \ge n / (80 \log n) \ge 400 \beta / p^2$ for sufficiently small $\eps > 0$. Other conditions follow by the choice of parameters.


	Let $V_C = U_0 \cup U_1 \cup \ldots \cup U_q$ denote the set of all vertices which are part of some chain. Consider a set $B \subseteq V(G) \setminus V_C$ of all vertices which have less than $|U_0|p/2$ neighbours in $U_0$. 
	By Lemma \ref{lemma:small_degree} and $|U_0| = t_0 \ge n / (80 \log n)$ we have $|B| < 320 np^2 / \log n = o(np^2)$. We start by greedily covering vertices in $B$ with vertex-disjoint triangles using $V_1$: for each vertex $v \in B$, sequentially, choose a triangle which contains $v$, has the other two vertices in $V_1$ and is disjoint from previously chosen triangles. We can indeed do that as no matter how such triangles are chosen by the end they occupy at most $3|B|$ vertices in $V_1$. Each vertex has $np/6$ neighbours in $V_1$, thus only a negligible portion of it is occupied. By Lemma \ref{lemma:edge} we can find an edge within remaining vertices which gives a desired triangle.

	Next, let $M \subseteq V(G) \setminus V_C$ be the set of all the unused vertices outside of chains, that is vertices which are not part of any triangle chosen in the previous step. Pick a maximal collection of vertex-disjoint triangles within $M$. By Lemma \ref{lemma:triangle} this leaves us with a subset $L \subseteq M$ of size at most $2\beta / p^2 = 2 \eps n / \log n$. We now cover $L$ with the help of $U_0$. Note that $L$ is significantly larger than the set $B$ thus we cannot apply the same greedy strategy -- if we are not careful we could exhaust all the neighbours of some vertex  before we cover it with a triangle. 

	We go around this using Theorem \ref{thm:hax}, similarly as in the proof of Lemma \ref{lemma:cascade}. For each $v \in L$ form a graph $G_v$ on the vertex set $U_0$ by joining two vertices iff they form a triangle with $v$. To cover $L$ using disjoint triangles with two endpoints in $U_0$ we need to chose an edge from each $G_v$ such that all these edges are pairwise disjoint. Theorem \ref{thm:hax} tells us that if for every $I \subseteq L$ we have
	\begin{equation} \label{eq:haxell}
		\tau(\bigcup_{v \in I} G_v) \ge 3|I|,
	\end{equation}
	then such a choice is possible. Equivalently, for every $I \subseteq L$ and for every $X \subseteq U_0$ of size $|X| < 3|I|$ it suffices to show that there is a triangle with one vertex in $I$ and the other two in $U_0 \setminus X$. 

	Suppose first $|I| \le t_0 p/12$. Then $|X| \le t_0 p/4$ and as every vertex in $L$ has at least $|U_0|p/2$ neighbours in $U_0$ (recall that we have already covered those that did no have this property in the first step) there exists an edge among those (at least) $t_0 p/4 \ge \beta / p$ ones outside of $X$ (by Lemma \ref{lemma:edge}). Otherwise, if $|I| \ge t_0 p/12 > 1 + 4 \beta$ then $|X| \le 3|I| \le 3|L| < t_0/2$ thus $|U_0 \setminus X| \ge t_0/2 \ge 2\beta / p^2$ and by Lemma \ref{lemma:triangle} there exists a triangle with one endpoint in $L$ and the other two in $U_0 \setminus X$. 

	To conclude, up to this point we have found disjoint triangles which cover all the vertices in $V(G) \setminus V_C$ and some in $U_0$. Let $L_0 \subseteq [t_0]$ be the set of indices of vertices in $U_0$ which were not used in the previous phase. Now for each $i = 0, \ldots, q-1$ let $L'_{i+1} \subseteq [t_{i+1}]$ be a subset of indices such that
	$$
		\bigcup_{j \in L_i} V(C_j^i) \cup \bigcup_{j \in L_{i+1}'} V(C_j^{i+1})
	$$
	contains a triangle-factor and set $L_{i+1} = [t_{i+1}] \setminus L_{i+1}'$. Such sets $L'_{i+1}$ are guaranteed to exists by the choice of chains and Lemma \ref{lemma:cascade}. All together, this gives us a triangle-factor of a subgraph induced by
	$$
		V_C \cup \bigcup_{i = 0}^{q-1} U_i \cup \bigcup_{j \in [t_q] \setminus L_q'} V(C_j^q).
	$$
	Finally, let $L_q = \{i_1, \ldots, i_m\} = [t_q] \setminus L_q'$ and note that $|L_q|$ is divisible by $3$. Therefore it suffices to show that for each $j \in \{0, \ldots, \tfrac{m}{3} - 1\}$ we have that the subgraph induced by $V(C_{3j}^q) \cup V(C_{3j+1}^q) \cup V(C_{3j + 2}^q)$ contains a triangle-factor. By Lemma \ref{lemma:triangle} there exists a triangle which traverses $R(C_{3j}^q)$, $R(C_{3j+1}^q)$ and $R(C_{3j+2}^q)$ (recall that $|R(C_j^q)| \ge \max\{1 + 4\beta, 2\beta /p^2\}$), which implies that a desired triangle-factor exists (see Observation \ref{obs:chain}). This concludes the proof.
\end{proof}

\section{Concluding remarks}

\begin{itemize}	
	\item It remains to determine if the $\log n$ factor in Theorem \ref{thm:main} is necessary. A conjecture of Krivelevich, Sudakov and Szabo \cite{krivelevich2004triangle} states that it can be omitted. As a first step it would be interesting to show that if $G$ is $(p, o(np^2))$-bijumbled then one can cover all but $n^{1 - \eps}$ vertices with disjoint triangles, for some $\eps > 0$. Generalising this to disjoint triangles which intersect different sets of removable vertices, one could potentially only need to construct $O(1)$ levels of chains---compared to $\Theta(\log n)$ levels---similarly as in a proof of Krivelevich \cite{krivelevich1997triangle}. This could lead to a solution of the conjecture.

	\item A straightforward modification of our argument shows that every $(p, \eps np^{r-1} / \log n)$-bijumbled graph with minimum degree $np/2$ contains a $K_r$-factor. This improves a result of Allen et al.\ \cite{allen2017powers} which requires $\beta = O(np^{3r/2})$, though it is fair to note that they show the existence of a much richer subgraph, namely the $r$-th power of a Hamilton cycle. Whether such a bound on $\beta$ is optimal is unclear as we do not know if it is the weakest (up to $\log n$ factor) condition which guarantees the existence even of a single $K_r$. 

	\item With a bit more effort we believe the same argument should also work for \emph{$(p, \beta)$-jumbled} graphs (see \cite{thomason1}), rather than the stronger bijumbled version. 


	\item It would be interesting to see if $\beta = o(np^2 / \log n)$ is also sufficient for the existence of any $2$-regular spanning subgraph, that is any collection of cycles with the total size $n$. A stronger result would be to improve a result of Allen et al.\ \cite{allen2017powers} and show that such $\beta$ ensures the square of a Hamilton cycle.

	\item The proof strategy can be seen as a `derandomisation' of a proof by Krivelevich \cite{krivelevich1997triangle}. One advantage of this approach compared to other absorbing-type proofs is that chains are very easy to construct and the effort goes into showing how to use them as absorbers. We hope this will make other problems on $H$-factors and, potentially, more general graphs, easier to tackle.
\end{itemize}

\paragraph{Acknowledgement.} The author would like to thank Anita Liebenau and Yanitsa Pehova for valuable discussions.

\bibliographystyle{abbrv}
\bibliography{references}

\end{document}